\documentclass[a4paper,11pt]{amsart}
\usepackage{amssymb,amscd,amsxtra,graphicx,psfrag}

\usepackage[pdftex]{hyperref}
\hypersetup{%
bookmarksnumbered=true,%
colorlinks=true,%
setpagesize=false,%
pdftitle={},%
pdfauthor={Takuzo Okada}}

\theoremstyle{plain}
\newtheorem{Thm}{Theorem}[section]
\newtheorem{Lem}[Thm]{Lemma}
\newtheorem{Cor}[Thm]{Corollary}
\newtheorem{Prop}[Thm]{Proposition}

\theoremstyle{definition}
\newtheorem{Def}[Thm]{Definition}
\newtheorem{Def-Lem}[Thm]{Definition-Lemma}

\newtheorem{Rem}[Thm]{Remark}

\newtheorem*{Ack}{Acknowledgments}

\theoremstyle{remark}

\newcommand{\Coker}{\operatorname{Coker}}

\newcommand{\Proj}{\operatorname{Proj}}
\newcommand{\prt}{\partial}

\newcommand{\Sing}{\operatorname{Sing}}
\newcommand{\Spec}{\operatorname{Spec}}

\newcommand{\mbA}{\mathbb{A}}

\newcommand{\mbC}{\mathbb{C}}

\newcommand{\mbP}{\mathbb{P}}

\newcommand{\mbZ}{\mathbb{Z}}

\newcommand{\mcL}{\mathcal{L}}
\newcommand{\mcM}{\mathcal{M}}

\newcommand{\mcO}{\mathcal{O}}

\newcommand{\mcQ}{\mathcal{Q}}

\newcommand{\mcX}{\mathcal{X}}

\newcommand{\mfm}{\mathfrak{m}}

\newcommand{\msp}{\mathsf{p}}
\newcommand{\msq}{\mathsf{q}}
\newcommand{\K}{\Bbbk}

\newcommand{\inj}{\hookrightarrow}

\newcommand{\CH}{\operatorname{CH}}

\title[Stable rationality of cyclic covers]{Stable rationality of cyclic covers \\ of projective spaces}
\author[Takuzo Okada]{Takuzo Okada}
\address{Department of Mathematics, Faculty of Science and Engineering\endgraf
Saga University, Saga 840-8502 Japan}
\email{okada@cc.saga-u.ac.jp}
\date{}

\begin{document}

\begin{abstract}
The main aim of this paper is to show that a cyclic cover of $\mbP^n$ branched along a very general divisor of degree $d$ is not stably rational provided that $n \ge 3$ and $d \ge n+1$.
This generalizes the result of Colliot-Th\'el\`ene and Pirutka.
Generalizations for cyclic covers over complete intersections and applications to suitable Fano manifolds are also discussed.
\end{abstract}

\maketitle


\section{Introduction}

In this paper we study stable rationality of cyclic covers of various varieties.
We first explain known non-rationality results on cyclic covers of $\mbP^n$.
Koll\'ar \cite{Kol1} proved that a wide range of cyclic covers over $\mbP^n$ is not rational and, for some of them, it is moreover proved that they are not birational to a conic bundle.  
Although the non-rationality result overlaps the avobe mentioned Koll\'ar's result, there are some results which focus more on birational nature of cyclic coverings.
It is proved in e.g.\ \cite{Puk}, \cite{Chel2} that a double cover of $\mbP^n$ branched along a divisor of degree $2n$ is birationally superrigid  for $n \ge 3$, which in particular implies that the variety is not birational to a Mori fiber space other than itself. 
It is also proved in \cite{Chel1} that a triple cover of $\mbP^{2n}$ branched along a divisor of degree $3n$ is birationally superrigid for $n \ge 2$.

Recently, non-stable-rationality of a double cover of $\mbP^3$ branched along a very general quartic is proved by Voisin \cite{Voisin} by introducing a new specialization method.
Beauville \cite{Beau} proved that a double cover of $\mbP^3$ branched along a very general sextic is not stably rational.
The idea has been amplified and applied to many varieties other than cyclic covers such as hypersurfaces \cite{CTP1,Totaro}, conic bundles \cite{HKT}, Fano $3$-folds and del Pezzo fibrations \cite{HT}, and quadric surface bundles \cite{HPT}.

It is proved by Colliot-Th\'el\`ene and Pirutka \cite{CTP2} that a cyclic cover of $\mbP^n$ branched along a very general divisor of degree $d$ is not stably rational if $d \ge n+1$, $n \ge 3$ and the covering degree is a prime number.
The first aim of this paper is to drop the assumption on the covering degree.

\begin{Thm} \label{mainthm}
A cyclic cover of $\mbP^n$ defined over $\mbC$ branched along a very general divisor of degree $d$ is not stably rational provided that $n \ge 3$ and $d \ge n+1$.
If in addition $d \ge n+2$ and the covering degree is not a power of $2$, then it is not birational to a conic bundle.
\end{Thm}

The main part of the proof is to prove the existence of a ``good" resolution of singularities of an inseparable cyclic covering space defined over an algebraically closed field of positive characteristic $p$, which will be discussed in Section \ref{sec:resol}.
Existence of such a good resolution is known for some restricted cases, for example, the case when the covering degree is $p$ (see \cite{CTP2}).
In this paper, we establish the existence of such a good resolution in full generality, that is, for inseparable cyclic covering of arbitrary degree over an arbitrary nonsingular variety (see Proposition \ref{prop:resol}).

The existence result of a good resolution in a general setting enables us to work with cyclic covers over varieties other than $\mbP^n$ and provides us with a simple proof of non-stable-rationality of many cyclic covers whose covering degrees are not too small.
The following generalization of Theorem \ref{mainthm} is motivated by the very recent work of Chatzistamatiou and Levine \cite{CL}, where, among other deep results, it is in particular proved that a very general Fano complete intersection of index $1$ is not stably rational.

\begin{Thm} \label{mainthmgen}
Let $X$ be a cyclic cover of a very general complete intersection of type $(d_1,\dots,d_r)$ in $\mbP^{n+r}$ branched along a very general divisor of degree $d \ge 2$ such that $n \ge 3$ and $d + d_1 + \cdots + d_r \ge n+r+1$.
Then $X$ is not stably rational.
If in addition $d + d_1 + \cdots + d_r \ge n+r+2$ and the covering degree is not a power of $2$, then $X$ is not birational to a conic bundle.
\end{Thm}

%

We explain a further application.

\begin{Thm} \label{mainthm2}
Let $n \ge 3, a, b, d$ be positive integer satisfying the following properties.
\begin{enumerate}
\item $a$ and $b$ are coprime.
\item $d$ is divisible by both $a$ and $b$, and $d/a, d/b \ge 2$. 
\item There is a prime number $p$ which divides $d/b$ but not $d/a$.
\item $n+a \le d < n + a + b$.
\end{enumerate}
Then a very general weighted hypersurface $X$ of degree $d$ in $\mbP (1^n,a,b)$, defined over $\mbC$, is a nonsingular Fano $n$-fold which is not stably rational. 
\end{Thm}

In dimension $3$, this recovers the result of \cite{HT} for Fano $3$-folds  $X_6 \subset \mbP (1,1,1,2,3)$, where the subscript indicates the degree.
By \cite{BK}, nonsingular Fano $4$-fold weighted hypersurfaces of index $1$ consist of $4$ families, $X_5 \subset \mbP^5$, $X_6 \subset \mbP (1^5,2)$, $X_8 \subset \mbP (1^5,4)$ and $X_{10} \subset \mbP (1^4,2,5)$, where the non-stable-rationality of $X_5 \subset \mbP^5$ is proved in \cite{Totaro}.
Following is a consequence of Theorems \ref{mainthm} and \ref{mainthm2}.

\begin{Cor} \label{maincor2}
A very general nonsingular Fano $4$-fold weighted hypersurface of index $1$ is not stably rational.
\end{Cor}


\begin{Ack}
The author is partially supported by JSPS KAKENHI Grant Number 26800019.
\end{Ack}

\section{Preliminaries} \label{sec:prelim}

In this section, we explain idea of proof of Theorem \ref{mainthm} by recalling necessary definitions and results.
For a variety $X$, we denote by $\CH_0 (X)$ the Chow group of $0$-cycles on $X$.
For universal $\CH_0$-triviality, we refer readers to the survey paper \cite{Pir} (ant the references therein) for details.

\begin{Def}
\begin{enumerate}
\item A projective variety $X$ defined over a field $k$ is {\it universally $\CH_0$-trivial} if for any field $F$ containing $k$, the degree map $\CH_0 (X_F) \to \mbZ$ is an isomorphism.
\item A projective morphism $\varphi \colon Y \to X$ defined over a field $k$ is {\it universally $\CH_0$-trivial} if for any field $F$ containing $k$, the push-forward map $\varphi_* \colon \CH_0 (Y_F) \to \CH_0 (X_F)$ is an isomorphism.
\end{enumerate}
\end{Def}

For the proof of Theorem \ref{mainthm}, we apply the following specialization argument.

\begin{Lem}[{\cite[Lemma 1.5]{CTP1}}]
If $X$ is a smooth, projective, stably rational variety, then $X$ is universally $\CH_0$-trivial.
\end{Lem}

\begin{Thm}[{\cite[Th\'eor\`eme 1.14]{CTP1}}] \label{thm:specialization}
Let $A$ be a discrete valuation ring with fraction field $K$ and residue field $k$, with $k$ algebraically closed.
Let $\mcX$ be a flat proper scheme over $A$ with geometrically integral fibers.
Let $X$ be the general fiber $\mcX \times_A K$ and $Y$ the special fiber $X \times_A k$.
Assume that there is a universally $\CH_0$-trivial proper birational morphism $Y' \to Y$ with $Y'$ smooth over $k$.
Let $\overline{K}$ be an algebraic closure of $K$.
Assume that there is a proper birational morphism $X' \to X$ smooth over $K$ such that $X'_{\overline{K}}$ is universally $\CH_0$-trivial.
Then $Y'$ is universally $\CH_0$-trivial.
\end{Thm}

Let $X$ be a cyclic cover of $\mbP^n$ or a very general complete intersection branched along a very general divisor of degree $d$ and let $m$ be the covering degree.
Theorem \ref{thm:specialization} enables us to work over an algebraically closed field of positive characteristic $p$, where we can apply Koll\'ar's construction of a suitable line bundle $\mcM \subset (\Omega_X^{n-1})^{\vee \vee}$ (Here $\vee \vee$ denotes the double dual).
Koll\'ar's construction will be explained in the next section.
The condition $d \ge n+1$ or $d+d_1+\cdots+d_r \ge n+r+1$ implies $H^0 (X,\mcM) \ne 0$.
We apply the following criterion for universal $\CH_0$-triviality.

\begin{Thm}[{\cite[Lemma 2.2]{Totaro}}] \label{thm:totaro}
Let $X$ be a smooth projective variety over a field $k$.
If $H^0 (X, \Omega_X^i) \ne 0$ for some $i > 0$, then $X$ is not universally $\CH_0$-trivial.
\end{Thm}

The above mentioned Koll\'ar's construction is valid only when the characteristic divides $m$, which forces $X$ to have singularities.
It is then important to construct  a universally trivial $\CH_0$ resolution $\varphi \colon Y \to X$ such that $\varphi^*\mcM \subset \Omega_Y^{n-1}$.
The latter implies $H^0(Y,\Omega_Y^{n-1}) \ne 0$, hence Theorem \ref{mainthm} (see Section \ref{sec:proof} for precise arguments).

\section{Singularities of inseparable coverings} \label{sec:sing}

We briefly recall Koll\'ar's construction of a suitable line bundle on a covering space and give an explicit description of singularities of such a space.
We refer readers to \cite[V.5]{Kol2} for details.
Then we introduce the notion of admissible critical point, which in most of the cases coincides with (almost) nondegenerate critical point introduced by Koll\'ar, but is important when the characteristic is $2$.

Throughout the present section, we keep the following setting. 
We work over an algebraically closed field $\K$ of positive characteristic $p$.
Let $Z$ be a smooth variety of dimension $n$ defined over $\K$, $m$ a positive integer such that $p \mid m$, $\mcL$ a line bundle on $Z$ and $s \in H^0 (Z, \mcL^m)$. 

\subsection{Koll\'ar's construction}

Let $U = \Spec (\oplus_{i \ge 0} \mcL^{-i})$ be the total space of the line bundle $\mcL$ and $\pi_U \colon U \to Z$ the natural morphism.
We denote by $y \in H^0 (U,\pi^*_U \mcL)$ the zero section and define
\[
Z [\sqrt[m]{s}] = (y^m - s = 0) \subset U.
\]
Set $X = Z [\sqrt[m]{s}]$ and let $\pi = \pi_U|_X \colon X \to Z$ be the covering morphism.

Let $\tau$ be a local generator of $\mcL$ at a point $\msq \in Z$ and $t = f \tau^m$ a local section of $\mcL$, where $f \in \mcO_{Z,\msq}$.
Let $x_1,\dots,x_n$ be local coordinates of $Z$ at $\msq$.
Then we define
\[
d (t) = \sum \frac{\prt f}{\prt x_i} t^m d x_i \in \mcL^m \otimes \Omega_Z^1.
\]
This is independent of the choices of local coordinates $x_1,\dots,x_n$ and the local generator $\tau$, and thus defines $d \colon \mcL^m \to \mcL^m \otimes \Omega_Z^1$.
For the section $s \in H^0 (Z, \mcL^m)$, we can view $d (s)$ as a sheaf homomorphism $d (s) \colon \mcO_Z \to \mcL^m \otimes \Omega_Z^1$.
By taking the tensor product with $\mcL^{-m}$, we obtain $d s \colon \mcL^{-m} \to \Omega_Z^1$.

\begin{Def}
We define
\[
\mcQ (\mcL,s) = \left(\det \Coker (ds)\right)^{\vee \vee}.
\]
\end{Def}

Note that $\mcQ (\mcL,s) \cong \mcL^m \otimes \omega_Z$ and, by \cite[V.5.5 Lemma]{Kol2}, there is an injection
\[
\rho^*\mcQ (\mcL,s) \inj (\Omega_X^1)^{\vee \vee}.
\]

We explain a local generator of $\pi^*\mcQ (\mcL,s)$.
Let $\msp \in X$ be a point and $\msq = \pi (\msp) \in Z$.
Let $x_1,\dots,x_n$ be local coordinates of $Z$ at $\msq$, $\tau$ a local generator and write $s = f \tau^m$.
We set
\[
\eta_i = \frac{d x_1 \wedge \cdots \widehat{d x_i} \wedge \cdots \wedge d x_n}{\prt f/\prt x_i}
\]
for $i = 1,\dots,n$.
Then $\eta_i = \pm \eta_j$ and they give a local generator of $\pi^*\mcQ (\mcL,s)$ which we denote by $\eta_{\msp}$ (see \cite[V.5.9 Lemma]{Kol2}).

\subsection{Description of singularities}

The singularities of $X = Z [\sqrt[m]{s}]$ can be understood in terms of critical points of $s$.

\begin{Def}
Let $s \in H^0 (Z,\mcL^m)$ be a section.
We say that $s$ has a critical point at $\msq \in Z$ if $d (s) \in H^0 (Z,\mcL^m \otimes \Omega_Z^1)$ vanishes at $\msq$.
We say that $s$ is {\it singular} at $\msq \in Z$ if it has a critical point at $\msq$ and $s$ vanishes at $\msq$.

Suppose that $s$ has a critical point at $\msq$.
Then we say that $s$ has a {\it nondegenerate critical point} at $\msq$ if in a suitable local coordinates $x_1,\dots,x_n$ and a local generator $\tau$ of $\mcL$, $s$ can be written as
\[
s = (\alpha + q (x_1,\dots,x_n) + (\text{higher degree terms})) \tau^m,
\]
in a neighborhood of $\msq$, where $\alpha \in \K$, and the quadratic part $q$ is nondegenerate.
If $p = 2$ and $n$ is odd, then a critical point is always degenerate.
In this case we say that $s$ has an {\it almost nondegenerate critical point} at $\msq$ if in a suitable choice of local coordinates $x_1,\dots,x_n$ and a local generator $\tau$ of $\mcL$, $s$ can be written as
\[
s = (\alpha + \beta x_1^2 + x_2 x_3 + x_4 x_5 + \cdots + x_{n-1} x_n + f) \tau^m,
\]
where $\alpha,\beta \in \K$, $f$ consists of monomials of degree at least $3$ and the coefficient of $x_1^3$ in $f$ is nonzero.
\end{Def}


It is easy to see that $\Sing X = \pi^{-1} (\operatorname{Crit} (s))$, where  $\operatorname{Crit} (s)$ denotes the set of critical points of $s$.
We give an explicit local description of singularities of $X$. 
Suppose that $s$ has a (almost) nondegenerate critical point at $\msq \in Z$.
Then, by a suitable choice of local coordinates, $s$ can be written as $s = \alpha + q + f$, where $\alpha \in \K$,
\begin{equation} \label{eq:nondegquad}
q =
\begin{cases}
x_1 x_2 + x_3 x_4 + \cdots + x_{n-1} x_n, & \text{if $n$ is even}, \\
\beta x_1^2 + x_2 x_3 + x_4 x_5 + \cdots + x_{n-1} x_n, & \text{if $n$ is odd},
\end{cases}
\end{equation}
and $f = f (x_1,\dots,x_n)$ consists of monomials of degree at least $3$.
Moreover, if $n$ is odd and $p \ne 2$, then we may assume $\beta = 1$, and if $n$ is odd and $p = 2$, then the coefficient of $x_1^3$ in $f$ is nonzero.
Then locally over $\msq$, $X$ is defined by
\begin{equation} \label{eq:initialdefeq}
y^m + \alpha + q + f = 0.
\end{equation}

\begin{Lem} \label{lem:descrsing}
Suppose $s$ has an $($alomost$)$ nondegenerate critical point at $\msq$.
Let $\msp \in X$ be a singular point such that $\pi (\msp) = \msq$.
Then in a suitable choice of local coordinates $x_1,\dots,x_n$ and $y$ of $X$ with origin at $\msp$, $X$ is defined by an equation of the form
\begin{equation} \label{eq:basicdefeq}
\gamma_e y^{e d} + \gamma_{e-1} y^{(e-1)d} + \cdots + \gamma_2 y^{2d} + y^d + q + f = 0,
\end{equation}
where $p \mid d$, $e \ge 1$, $\gamma_i \in \K$ and $q,f$ are as above.
\end{Lem}

\begin{proof}
If $\alpha = 0$, then the claim follows obviously by setting $e = 1$ and $d = m$.
Suppose that $\alpha \ne 0$.
We write $m = d e$, where $d$ is a power of $p$ and $e$ is not divisible by $p$.
Then the claim follows by replacing $y \mapsto y - \sqrt[m]{\alpha}$ and then scaling $y$ in the defining equation \eqref{eq:initialdefeq}.
\end{proof}

\begin{Rem} \label{rem}
Let $\msp \in X$ be as in Lemma \ref{lem:descrsing} and $\eta_{\msp}$ the local generator of $\pi^*\mcQ (\mcL,s)$.
In the next section, we choose $\eta_n$ as a representative of $\eta_{\msp}$.
Then
\[
\eta_{\msp} = \frac{d x_1 \wedge \cdots \wedge d x_{n-1}}{x_n + \cdots},
\]
where the omitted term in $x_n + \cdots$ is a sum of monomials of degree at least $2$.
\end{Rem}

\subsection{Admissible critical points}

In the next section, we will construct a resolution of singularities of the germ $\msp \in X$ given in Lemma \ref{lem:descrsing}.
However we have a problem in constructing a desired resolution when  $n$ is odd, $p = 2$, $d > 2$ and $\beta = 0$ in the defining equation \eqref{eq:basicdefeq}.
Under the assumption that $n$ is odd and $p = 2$, this bad situation happens if and only if $\beta = 0$ and either $2^2 \mid m$ or $m > 2$ and $s$ is singular at $\msq$ (i.e.\ $\alpha = 0$). 
We explain that, in this case, we can avoid the situation $\beta = 0$ under a mild condition.

We assume that $s$ has a critical point at $\msq$.
Moreover we assume that $n$ is odd, $p = 2$ and either $2^2 \mid m$ or $m > 2$ and $s$ is singular at $\msq$.
Let $x_1,\dots,x_n$ be local coordinates of $Z$ at $\msq$ and $\tau$ a local generator of $\mcL$ at $\msq$. 
Then $s = g (x_1,\dots,x_n) \tau^m$ and we can write $g = \alpha + g_1 + g_2 + \cdots$, where $\delta \in \K$ and $g_i$ is homogeneous of degree $i$.
Note that $g_1 = 0$ since $s$ has a critical point at $\msq$.
If we choose another local generator $\tau'$ of $\mcL$ and write $s = g' {\tau'}^{m}$, then $g' = u^m g$ where $u = u (x_1,\dots,x_n)$ is a unit of $\mcO_{Z,\msq}$.
We can write $u = \lambda + h$ for some nonzero $\lambda \in \K$ and $h \in \mfm_{\msp}$.
We note that the ground field is of characteristic $2$ and $2 \mid m$.
It follows that $g'_2 = \lambda^m g_2$ since either $2^2 \mid m$ or $m > 2$ and $\alpha = 0$.
Thus, the quadratic part $g_2$ does not depend on the choice of local generator (up to a multiple by a non-zero constant).

\begin{Def}
Under the assumption, we define $T (s,\msq)$ to be 
\[
(g_2 = 0) \subset \mbP^{n-1} = \Proj \K [x_1,\dots,x_n]
\]
and call it the {\it quadric associated with the critical point} $\msq$ of $s$.
\end{Def}

Note that the isomorphism class of $T(s,\msq)$ does not depend on the choice of local coordinates. 

\begin{Lem} \label{lem:admcr}
Assume that $n = \dim Z$ is odd, $p = 2$ and $m > 2$.
Let $V \subset H^0 (Z,\mcL^m)$ be a finite dimensional vector space and assume that the restrictions map
\[
V \to \mcL^m \otimes (\mcO_{Z,\msq}/\mfm_{\msq}^3)
\]
is surjective for any point $\msq \in Z$, where $\mfm_{\msq}$ is the maximal ideal of $\mcO_{Z,\msq}$.
Then, a general $s \in V$ has only critical points $\msq$ such that $T (s,\msq)$ is nonsingular.
\end{Lem}

\begin{proof}
Proof will be done by counting dimensions.
Let $V^{\operatorname{cr}}_{\msq}$ be the set of sections in $V$ which has a critical point at $\msq$.
By the surjectivity of the restriction map, we see that $V_{\msq}^{\operatorname{cr}}$ is of codimension $n$ in $V$.
Again, by the surjectivity of the restriction map, the set $V^{\operatorname{sing}}_{\msq}$ of sections $s$ in $V^{\operatorname{cr}}_{\msq}$ for which $T(s,\msq)$ is singular is of codimension $1$ in $V^{\operatorname{cr}}$.
Therefore, the assertions follows by the dimension counting argument.
\end{proof}

We finish this section by introducing a unified notion of (almost) nondegenerate critical point and the singularity of $T(s,\msq)$, and then improving the description in Lemma \ref{lem:descrsing}.

\begin{Def}
We say that $s \in H^0 (Z,\mcL^m)$ has an {\it admissible critical point} at $\msq \in Z$ if $s$ satisfies one of the following:
\begin{enumerate}
\item Either $n$ is even or $n$ is odd and $p \ne 2$, and $s$ has a nondegenerate critical point at $\msq$.
\item $n$ is odd, $p = 2$, $m = 2$ and $s$ has an almost nondegenerate critical point at $\msq$.
\item $n$ is odd, $p = 2$, $m \ne 2$, $2^2 \nmid m$ and $s$ has an almost nondegenerate critical point at $\msq$ but not singular at $\msq$.
\item $n$ is odd, $p = 2$, $m \ne 2$, $2^2 \nmid m$, $s$ has an almost nondegenerate critical point at $\msq$, $s$ is singular at $\msq$ and $T (s,\msq)$ is nonsingular.
\item $n$ is odd, $p = 2$, $2^2 \mid m$, $s$ has an almost nondegenerate critical point at $\msq$ and $T(s,\msq)$ is nonsingular.
\end{enumerate}
\end{Def}

\begin{Lem} \label{lem:descrsingadd}
Suppose $s$ has an admissible critical point at $\msq$.
Let $\msp \in X$ be a singular point such that $\pi (\msp) = \msq$ and let
\[
\gamma_e y^{ed} + \gamma_{e-1} y^{(e-1)d} + \cdots + \gamma_2 y^{2d} + y^d + q + f = 0,
\]
be the defining equation of the germ $\msp \in X$ given in \emph{Lemma \ref{lem:descrsing}}. 
If $n$ is odd, $p = 2$ and $d > 2$, then $\beta \ne 0$, that is, $x_1^2$ appears in $q$ with non-zero coefficients.
\end{Lem}

\begin{proof}
Suppose that $n$ is odd and $p = 2$.
Then the case $d > 2$ happens if and only if either $2^2 \mid m$ or $m \ne 2$ and $s$ is singular at $\msq$.
In this case $T (s,\msq)$ is nonsingular and this implies that $\beta \ne 0$.
\end{proof}

\section{Good resolution of inseparable coverings} \label{sec:resol}

The aim of this section is to prove the following.

\begin{Prop} \label{prop:resol}
Let $\msp \in X = Z[\sqrt[m]{s}]$ be a germ of an isolated singularity corresponding to an admissible critical point of $s$ at $\pi (\msp)$.
Assume that $n = \dim Z = \dim X \ge 3$. 
Then there exists a universally $\CH_0$-trivial resolution $\varphi \colon Y \to X$ of singularities such that the pullback $\varphi^*\eta_{\msp}$ is a regular section of $\Omega_Y^{n-1}$.
\end{Prop}

We actually prove the following for the resolution $\varphi \colon Y \to X$ in Proposition \ref{prop:resol}: The exceptional divisor $\varphi^{-1} (\msp) = \cup E_i$ is a chain of nonsingular rational varieties.
More precisely, each irreducible component $E_i$ is a nonsingular rational variety, $E_i \cap E_j= \emptyset$ if $|i-j| > 1$ and $E_i \cap E_{i+1}$ is an $(n-2)$-dimensional nonsingular quadric.
Universal $\CH_0$-triviality follows from Proposition \ref{propCTP1} and Lemma \ref{lemCTP2} below.

The resolution will be constructed by successive blowups at each stage of the unique singular point.
The exceptional divisor of the last blowup is either a projective space or a nonsingular quadric.
The other irreducible component $E_i$ is the blowup $\operatorname{Bl}_{v_i} (Q_i)$ of a quadric cone $Q_i$ at its vertex $v_i$ and $E_i$ intersects $E_{i-1}$ along the exceptional divisor of $E_i = \operatorname{Bl}_{v_i} Q_i \to Q_i$.

\begin{Prop}[{\cite[Proposition 1.7]{CTP1}}] \label{propCTP1}
Let $\varphi \colon Y \to X$ be a projective morphism such that for any scheme point $x$ of $X$, the fiber $\varphi^{-1} (x)$, considered as a variety over the residue field $k (x)$, is universally $\CH_0$-trivial.
Then $\varphi$ is universally $\CH_0$-trivial.
\end{Prop}

\begin{Lem}[{\cite[Lemma 2.4]{CTP2}}] \label{lemCTP2}
Let $X = \cup X_i$ be a projective, reduced, geometrically connected variety such that each $X_i$ is universally $\CH_0$-trivial and geometrically irreducible, and each intersection $X_i \cap X_j$ is either empty or has a $0$-cycle of degree $1$.
Then $X$ is universally $\CH_0$-trivial.
\end{Lem}

The desired resolution $\varphi \colon Y \to X$ will be constructed by successive blowups at a point.
We explain notation necessary to describe those blowups.

Let $\mbA^{n+1}$ be an affine space with affine coordinates $x_1,\dots,x_n$ and $y$ and $V \subset \mbA^{n+1}$ a hypersurface defined by $f (x_1,\dots,x_n,y) = 0$ passing through the origin $\msp \in \mbA^{n+1}$.
We decompose $f = f_d + f_{> d}$, where $f_d$ and $f_{>d}$ consists of monomials of degree $d$ and at least $d+1$, respectively (with respect to the usual grading $\deg x_i = \deg y = 1$).
Let $\sigma \colon W \to V$ be the blowup of $V$ at $\msp$.
The varieties $W$ are covered by the standard affine charts which we denote by $W_{x_1},\dots,W_{x_n}$ and $W_y$.
Explicitly, the $x_i$-chart and the $y$-chart are described as
\[
\begin{split}
W_{x_i} &= (f (x_1 x_i,\dots,x_{i-1} x_i,x_i,x_{i+1} x_i,\dots,x_n x_i,y x_i)/x_i^d = 0) \subset \mbA^{n+1}, \\
W_y &= (f (x_1,y,\dots,x_n y,y)/y^d = 0) \subset \mbA^{n+1},
\end{split}
\]
The exceptional divisor $E$ is isomorphic to $(f_d = 0) \subset \mbP^n$, where $\mbP^n$ is a projective space with coordinates $x_1,\dots,x_n$ and $y$, and on $W_{x_i}$ and $W_y$, $E$ is defined by $x_i = 0$ and $y = 0$, respectively.
If $\msp \in V$ is a germ of an isolated singularity, then the singular locus of $W$ is entirely contained in $E$, hence, for example, on the $y$-chart $W_y$, $\Sing (W_y)$ is contained in $(y = 0)$.

\subsection{Case: Either $n$ is even or $n$ is odd and $p \ne 2$}

By Lemma \ref{lem:descrsing}, we have
\[
\msp \in X \cong o \in (y^d + q + f = 0) \subset \mbA^{n+1},
\]
where $p \mid d$, $q$ is of the form \eqref{eq:nondegquad} and $f$ is a sum of monomials $x_1^{i_1} \cdots x_n^{i_n}$ with $\sum i_k \ge 3$ and $y^j$ with $j \ge 2 d$.
The existence of the desired resolution $\varphi \colon Y \to X$ is already constructed in \cite[V.5.10 Proposition]{Kol2} and we briefly explain the construction.

Let $\sigma_1 \colon X_1 \to X$ be the blowup at $\msp$.
If $d = 2$ (in this case $p = 2$), then $X_1$ is nonsingular and the $\sigma_1$-exceptional divisor $E_1$ is a nonsingular quadric.
Suppose that $d \ge 3$.
It is easy to see that the $x_i$-chart $(X_1)_{x_i}$ is nonsingular for any $i$ and
\[
(X_1)_y \cong (y^{d-2} + q + \cdots = 0) \subset \mbA^{n+1}.
\]
Here we omit higher degree terms in defining equations simply because these are not necessary to analyze singularities.
We have $\Sing X_1 = \{\msp_1\}$, where $\msp_1$ is the origin of $(X_1)_y$. 
The $\sigma_1$-exceptional divisor $E_1$ is isomorphic to $(q=0) \subset \mbP^n$, which is a quadric cone with vertex $\msp_1$.
Let $\sigma_2 \colon X_2 \to X_1$ be the blowup at $\msp_1$.
We see that $\Sing (X_2) = \{\msp_2\}$, where $\msp_2$ is the origin of the $y$-chart $(X_2)_y$ of the blowup $\sigma_2$ and the $\sigma_2$-exceptional divisor $E_2$ is again a quadric cone with vertex $\msp_2$.
The proper transform $(\sigma_2)_*^{-1} E_1$ is the blowup of $E_1$ at the vertex $\msp_1$, so that it is a nonsingular rational variety.
Moreover the intersection $((\sigma_2)_*^{-1} E_1) \cap E_2$ is the fiber of the blowup $(\sigma_2)_*^{-1} E_1 \to E_1$, which is a nonsingular quadric, and it does not contain the point $\msp_2$ (which will be the center of the next blowup).

Iterating this process, we obtain a resolution $\varphi \colon X_k \to X$ of singularities.
However the exceptional divisor $E_k$ of the last blowup is still a quadric cone with vertex $\msp_k$ while all the other exceptional divisors are nonsingular. 
We blowup $X_k$ once more at $\msp_k$ and let $\varphi \colon Y \to X$ be the resulting morphism.
By construction, $\varphi^{-1} (\msp) = \cup E_i$ is a chain of nonsingular rational varieties and $E_i \cap E_{i+1}$ is a nonsingular quadric.
Then each exceptional divisor of the resulting birational morphism $Y \to X$ is a nonsingular rational variety.
The fact that $\varphi^*\eta_{\msp}$ does not have a pole along any exceptional divisor is proved in \cite[V.5.10]{Kol2}.
Therefore Proposition \ref{prop:resol} is proved in this case.

\subsection{Case: $n$ is odd and $p = 2$}

In this case, $n$ is odd, $p = 2$ and 
\[
\msp \in X \cong (y^d + q + \gamma x_1^3 + c + f = 0) \subset \mbA^{n+1},
\]
where $d, q, c, f$ are as follows:
\begin{itemize}
\item $d$ is even and $\gamma \in \K$ is nonzero.
\item $q = \beta x_1^2 + x_2 x_3 + x_4 x_5 + \cdots + x_{n-1} x_n$, where $\beta \in \K$ is $1$ if $d \ne 2$.
\item $c = c (x_1,\dots,x_n)$ is homogeneous of degree $3$ and it does not contain $x_1^3$.
\item $f = f (x_1,\dots,x_n,y)$ is a sum of monomials $x_1^{i_1} \cdots x_n^{i_n} y^j$ satisfying either $\sum i_k \ge 4$ or $i_1=\cdots=i_n=0$ and $j \ge 2 d$.
\end{itemize}

Note that this follows from Lemmas \ref{lem:descrsing} and \ref{lem:descrsingadd}.

If $d = 2$, then, by \cite[Proposition V.5.2]{Kol2}, the blowup $\sigma \colon X_1 \to X$ at $\msp$ resolves the singularities of $X$, its exceptional divisor $E_1$ is a quadric cone with vertex $\msp_1$, and $\sigma_1^*\eta$ does not have a pole along $E_1$.
If we blowup up $Y \to X_1$ at $\msp_1$, then the proper transform of $E_1$ is a nonsingular variety, hence we obtains the desired resolution.

In the following, we assume $d \ne 2$.

\begin{Def}
Let $h = h (x_1,\dots,x_n,y)$ be a polynomial, and $a,b$ be positive integers.
We say that $h$ is of {\it type} $(a ; b)$ if $h$ is a linear combination of monomials $x_1^{i_1} \cdots x_n^{i_n} y^j$ satisfying either $\sum i_k \ge a$ or $i_1=\cdots =i_n = 0$ and $j \ge b$.

For a polynomial $h = h (x_1,\dots,x_n,y)$, we define
\[
h^* := h (x_1 y, \dots,x_n y,y).
\]
\end{Def}
Note that if $h$ is of type $(a ; b)$, then $h^*$ is also of type $(a ; b)$.

Set $X_0 = X$, $\msp_0 = \msp \in X$ and $f_0 = f$ which is of type $(4;2d)$.

\begin{Lem} \label{lem:admresol1}
For $i = 1,2,\dots,d/2$, there exists a sequence $\sigma_i \colon X_i \to X_{i-1}$ of blowups at a point $\msp_{i-1} \in X_{i-1}$ with the following properties.
\begin{enumerate}
\item For any $i$, $\Sing X_i = \{\msp_i\}$.
\item For $i \ge 1$, the $\sigma_i$-exceptional divisor $E_i$ is a quadric cone with vertex $\msp_i$.
\item For $i = 2,\dots,k$, the proper transform $(\sigma_i)_*^{-1} E_{i-1}$ does not contain $\msp_i$.
\item There is an isomorphism
\[
\msp_{d/2} \in X_{d/2} \cong o \in (y^{d/2} (1+x_1) + q + y^{d/2} (\ell + g) + y^{d/2} h = 0) \subset \mbA^{n+1},
\]
where $\ell$ is a linear combination of $x_2, \dots, x_n$, $g = g (x_1,\dots,x_n)$ is a sum of monomials of degree at least $2$ and $h$ is of type $(2;d/2)$.
\end{enumerate}
\end{Lem}

\begin{proof}
Let $\sigma_1 \colon X_1 \to X_0$ be the blowup at $\msp_0$.
Clearly the $x_i$-chart $(X_1)_{x_i}$ is nonsingular for $i = 2,\dots,n$, and the $x_1$-chart $(X_1)_{x_1}$ is also nonsingular because of the presence of $x_1^2$ in $q$.
We have
\[
(X_1)_y \cong (y^{d-2} + q + \gamma y x_1^3 + y c + y^2 (f^*/y^4) = 0) \subset \mbA^{n+1}.
\]
We see that $\Sing X_1 = \{\msp_1\}$, where $\msp_1$ is the origin of the $y$-chart $(X_1)_y$, and the $\sigma_1$-exceptional divisor $E_1$ is a quadric cone with vertex $\msp_1$ since $x_1^2 \in q$.
We set $f_1 = f^*/y^4$ which is of type $(4;2d-4)$.

Let $\sigma_2 \colon X_2 \to X_1$ be the blowup at $\msp_1$.
As before, $(X_2)_{x_i}$ is nonsingular for any $i$ and
\[
(X_2)_y \cong (y^{d-4} + q + \gamma y^2 x_1^3 + y^2 c + y^4 (f^*_1/y^4) = 0) \subset \mbA^{n+1}.
\]
We see that $\Sing X_2 = \{\msp_2\}$, where $\msp_2$ is the origin of $(X_2)_y$, and $E_2$ is a quadric cone with vertex $\msp_2$.
We set $f_2 = f^*_1/y^4$ which is of type $(4;2d-8)$. 

We repeat this process $d/2$ times and we obtain a sequence of blowups 
\[
X_{d/2} \xrightarrow{\sigma_{d/2}} X_{d/2 -1} \to \cdots \xrightarrow{\sigma_1} X_0.
\]
We set $f_{d/2} = f^*_{d/2-1}/y^4$ which is of type $(4;0)$.
The variety $X_{d/2}$ is nonsingular along the $x_i$-charts of the blowup $\sigma_{d/2}$ for any $i$ and 
\[
(X_{d/2})_y \cong (1 + q + \gamma y^{d/2} x_1^3 + y^{d/2} c + y^d f_{d/2} = 0) \subset \mbA^{n+1}.
\]
We see that $\Sing X_{d/2} = \{\msp_{d/2}\}$, where $\msp_{d/2} = (1,0,\dots,0)$ on the $y$-chart, and the $\sigma_{d/2}$-exceptional divisor $E_{d/2}$ is a quadric cone with vertex $\msp_{d/2}$.

Now we make a coordinate change $x_1 \mapsto x_1 + 1$, which makes $\msp_{d/2}$ the origin.
After the coordinate change, we have
\[
(X_{d/2})_y \cong (\gamma y^{d/2} (1 + x_1) + q + \cdots = 0) \subset \mbA^{n+1},
\]
where the omitted term in the above equation is
\[
y^{d/2} (\gamma x_1^3 + \gamma x_1^2 + c(x_1+1,x_2,\dots,x_n)) + y^d f_{d/2} (x_1+1,x_2\dots,x_n,y)
\]
and we can write it as $y^{d/2} (\ell + g) + y^d h$, where $\ell$ is a linear combination of $x_2,\dots,x_n$, $g = g (x_1,\dots,x_n)$ is a sum of monomials of degree at least $2$ and $h = h (x_1,\dots,x_n,y)$.
By scaling $y$, we have the desired description of $\msp_{d/2} \in X_{d/2}$.

Finally we check (2).
On the $y$-chart $(X_{i-1})_y$ of the blowup $\sigma_{i-1}$, the exceptional divisor $E_{i-1}$ is defined by $y = 0$.
Since the singular point $\msp_i \in X_i$ is contained in the $y$-chart of $\sigma_i$, we verify (2).
\end{proof}

We need to construct a resolution of $\msp_{d/2} \in X_{d/2}$.

\begin{Def}
Let $k$ be a positive integer.
We say that a germ $\msq \in V$ of an $n$-dimensional isolated hypersurface singularity is of {\it type} $k$ if 
\[
\msq \in V \cong o \in (y^k (1 + x_1) + q + y^k (\ell + g) + y^{2k} h = 0) \subset \mbA^{n+1},
\]
where $\ell$ is a linear combination of $x_2,\dots,x_n$, $g$ is a sum of monomials $x_1^{i_1} \cdots x_n^{i_n} y^j$ satisfying $i_1+\cdots+i_n \ge 2$ and $h = h (x_1,\dots,x_n,y)$.
\end{Def}

\begin{Lem} \label{lem:admresol2}
Let $\msq_0 \in V_0$ be a germ of type $2k$ with $k \ge 1$.
Then, for $i = 1,2,\dots,k$, there is a sequence $\tau_i \colon V_i \to V_{i-1}$ of blowups at a point $\msq_{i-1} \in V_{i-1}$ with the following properties.
\begin{enumerate}
\item For any $0 \le i \le k-1$, $\Sing V_i = \{\msq_i\}$.
\item If $k \ge 2$, then $\Sing V_k = \{\msq_k\}$, and if $k = 1$, then $V_k$ is nonsingular.
\item For $1 \le i \le k$, the $\tau_i$-exceptional divisor $F_i$ is a quadric cone with vertex $\msq_i$.
\item For $2 \le i \le k$, the proper transform $(\tau_i)_*^{-1} F_{i-1}$ does not contain $\msq_i$.
\item If $k \ge 2$, then the germ $\msq_k \in V_k$ is of type $k$.
\end{enumerate}
\end{Lem}

\begin{proof}
Let $\tau_1 \colon V_1 \to V_0$ be the blowup at $\msq_0$.
We see that $(V_1)_{x_i}$ is nonsingular for any $i$ and
\[
(V_1)_y = (y^{2k-2} (1 + y x_1) + q + y^{2k-1} \ell + y^{2k} g_1 + y^{4k-2} h_1 = 0) \subset \mbA^{n+1},
\]
where $g_1 = g^*/y^2$ and $h_1 = h^*$.
Then $\Sing (V_1) = \{\msq_1\}$, where $\msq_1$ is the origin of $(V_1)_y$, and the $\tau_1$-exceptional divisor $F_1$ is a projective quadric cone with vertex $\msq_1$.

We repeat this procedure successively: 
After $(\tau_i,V_i,\msq_i,F_i,g_i,h_i)$ is constructed, $\tau_{i+1} \colon V_{i+1} \to V_i$ is the blowup at $\msq_i$, $\msq_{i+1}$ is the origin of the $y$-chart, $F_{i+1}$ is the $\tau_{i+1}$-exceptional divisor and $g_{i+1} = g_i^*/y^2$, $h_{i+1} = h^*$. 

At the $k$th stage, we have
\[
(V_{k})_y \cong (1+y^{k} x_1 + q + y^{k} \ell + y^{2 k} g_k + y^{2 k} h_k = 0) \subset \mbA^{n+1}.
\]
Let $\msq_k$ be the point $(1,0,\dots,0)$ of the $y$-chart $(V_k)_y$.
The $\tau_k$-exceptional divisor $F_k$ is a quadric cone with vertex $\msq_k$.
We see that $\Sing V_k \subset \{\msq_k\}$, and $V_k$ is nonsingular if and only if $k = 1$.

We replace $x_1$ with $x_1+1$, which makes $\msq_k$ the origin of the $y$-chart.
Then we have
\[
(V_k)_y \cong (y^k (1+x_1) + q + y^k \ell + y^{2k} h' = 0) \subset \mbA^{n+1},
\]
for some $h' \in \K [x_1,\dots,x_n,y]$.
This shows that $\msq_k \in V_k$ is a germ of type $k/2$. 
Finally, (4) follows since we keep working on the $y$-charts.
\end{proof}

\begin{Lem} \label{lem:admresol3}
Let $\msq_0 \in V_0$ be a germ of type $2k+1$ with $k \ge 1$.
Then, for $i = 1,2,\dots,k$, there is a sequence $\tau_i \colon V_i \to V_{i-1}$ of blowups at a point $\msq_{i-1} \in V_{i-1}$ with the following properties.
\begin{enumerate}
\item For any $0 \le i \le k-1$, $\Sing V_i = \{\msq_i\}$ and $V_k$ is nonsingular.
\item For $1 \le i \le k$, the $\tau_i$-exceptional divisor $F_i$ is a quadric cone with vertex $\msq_i$.
\item For $2 \le i \le k$, the proper transform $(\tau_i)_*^{-1} F_{i-1}$ does not contain $\msq_i$.
\end{enumerate}
\end{Lem}

\begin{proof}
We repeat the same blowing-up procedure as in the proof of Lemma \ref{lem:admresol2} replacing $2k$ by $2k+1$.
Let $\tau_i \colon V_i \to V_{i-1}$, $i = 1,2,\dots,k$ be the sequence of blowups at each singular point $\msq_{i-1} \in V_{i-1}$.
For the $y$-chart $(V_k)_y$ of the blowup $\tau_k$, we have
\[
(V_k)_y \cong y (1 + y^k x_1) + q + y^{k+1} \ell + y^{2k+1} g_k + y^{2k+2} h_k = 0) \subset \mbA^{n+1}.
\]
Since the singular locus of $(V_k)_y$ is contained in the exceptional divisor $(y = 0)$ of $\tau_k$, it is easy to see that $V_k$ is nonsingular.
The other assertions follows from the arguments in the proof of Lemma \ref{lem:admresol2}.
\end{proof}

We can conclude the existence of a resolution of $\msp \in X$.
We have a sequence $X_{d/2} \to \cdots \to X_0 = X$ of blowups.
We write $d/2 = 2^a d'$, where $a \ge 0$ and $d' \ge 1$ is odd. 
We apply Lemma \ref{lem:admresol2} recursively ($a$ times, to be precise) and we obtain a sequence $V_l \to \cdots \to V_0 = X_{d/2}$ of blowups.
If $d' = 1$, then $V_l$ is nonsingular, and if $d' > 1$, then $V_l$ has a unique singular point $\msq_l$ such that $\msq_l \in V_l$ is of type $d'$.
In the former case we set $V = V_l$.
In the latter case let $V \to \cdots \to V_l$ be a sequence of blowups given in Lemma \ref{lem:admresol3}.
Note that each exceptional divisor is a quadric cone and each blowup center is the vertex of the cone, so that the singularity of each exceptional divisor is also resolved the exceptional divisor, denoted by $F$, of the last blowup.
Let $Y \to V$ be the blowup at the vertex of $F$.
Then each exceptional divisor of the resulting morphism $\varphi \colon Y \to X$ is a nonsingular rational variety.
The fact that $\varphi^{-1} (\msp)$ is a chain and the intersections are nonsingular quadric follows from of Lemma \ref{lem:admresol1}.(3), Lemma \ref{lem:admresol2}.(4) and Lemma \ref{lem:admresol3}.(3).

It remains to check the lifting property of $\eta = \eta_{\msp}$.
As in Remark \ref{rem},
\[
\eta = \frac{d x_1 \wedge \cdots \wedge d x_{n-1}}{x_{n-1} + \cdots}.
\]
Note that the omitted terms consists of monomials of degree at least $2$ and we will omit the corresponding term after pulling back to the blowups of $X$.
In order to show that $\varphi^*\eta$ is a regular section of $\Omega_Y^{n-1}$, we introduce the following notion.

\begin{Def}
We define
\[
\theta_i = x_i \frac{d x_1 \wedge \cdots \wedge \widehat{d x_i} \wedge \cdots \wedge d x_{n-1} \wedge dy}{x_{n-1} + \cdots},
\]
for $i = 1,\dots,n$ and
\[
\theta_y = y \frac{d x_1 \wedge \cdots \wedge d x_{n-1}}{x_{n-1} + \cdots}.
\]
We say that a rational $(n-1)$-form $\omega$ is of type {\it theta} if it can be written as
\[
\omega = f_1 \theta_1 + \cdots + f_{n-1} \theta_{n-1} + g \theta_y,
\]
for some $f_1,\dots,f_{n-1},g \in \K [x_1,\dots,x_n,y]$. 
\end{Def}

Let $\sigma_1 \colon X_1 \to X$ be the first blowup.
On the $y$-chart $(X_1)_y$, we have
\[
\sigma_1^* \eta = y^{n-3} (\theta_1 + \cdots + \theta_{n-1} + \theta_y),
\]
and it is of type theta since $n \ge 3$.
Hence $\sigma^*_1 \eta$ does not have a pole along $E_1$ because $x_{n-1} + \cdots$ does not vanish entirely along $E_1$.
Let $\sigma$ be the blowup of $\mbA^{n+1}$ at the origin.
Then, on the $y$-chart of the blowup, we have
\[
\sigma^*\theta_i = y^{n-3} \theta_i, \text{ and } \sigma^*\theta_y = y^{n-2} (\theta_1+\cdots+\theta_{n-1}+\theta_y).
\]
This means that, at each blowup, the pullback of an $(n-1)$-form of type theta is again type theta.
Therefore $\varphi^*\eta$ is of type theta, which implies that $\varphi^*\eta$ is a regular section of $\Omega_Y^{n-1}$, completing the proof of Proposition \ref{prop:resol}.

\section{Proof of Main Theorems} \label{sec:proof}

\begin{proof}[Proof of \emph{Theorems \ref{mainthm}} and \emph{\ref{mainthmgen}}]
Let $Z$ be a very general complete intersection of type $(d_1,\dots,d_r)$ in $\mbP^{n+r}$.
We understand that $Z = \mbP^n$ when $r = 0$.
Moreover we assume $d_i \ge 2$ for any $i$ when $r > 0$.
By the assumptions in Theorems \ref{mainthm} and \ref{mainthmgen}, we have $d + d_1 + \cdots + d_r \ge n+r+1$.
Let $X$ be a cyclic cover of $Z$ branched along a very general divisor $D$ of degree $d$.
Let $m$ be the covering degree of $\pi \colon X \to Z$.
We choose and fix a prime number $p$ dividing $m$.
By Theorem \ref{thm:specialization}, it is enough to show that, over an algebraically closed field $\K$ of characteristic $p$, $X$ admits a universally $\CH_0$-trivial resolution $\varphi \colon Y \to X$ such that $Y$ is not universally $\CH_0$-trivial.

We see that $X \cong Z [\sqrt[m]{f}]$, where $f \in H^0 (Z,\mcO_Z (d))$ is the section defining $D$.
If $d > 2$, then the restriction map
\[
H^0 (\mbP^n,\mcO_Z (d)) \to \mcO_Z (d) \otimes (\mcO_{Z,\msq}/\mfm_{\msq}^4)
\]
is surjective for any point $\msq \in Z$.
Note that, in case $r = 0$, the inequality $d > 2$ is automatically satisfied since $d \ge n+1$ and $n \ge 3$.
It follows from \cite[V.5]{Kol2} and Lemma \ref{lem:admcr}, a general $f \in H^0 (\mbP^n, \mcO_{\mbP^n})$ has only admissible critical points.
It remains to consider the case when $r > 0$ and $d = 2$.
However, the same conclusion holds from \cite[Lemma 7.6]{CL} in this case as well.

Let $\varphi \colon Y \to X$ be the universally $\CH_0$-trivial resolution obtained in  Proposition \ref{prop:resol}.
Then $\varphi^* \mcM \inj \Omega_Y^{n-1}$, where $\mcM = \pi^* \mcQ (\mcO_Z (d),f)$.
We have $H^0 (Y,\Omega_Y^{n-1}) \ne 0$ since 
\[
\mcM \cong \pi^*(\omega_Z \otimes \mcO_{\mbP^n} (d)) \cong \mcO_X (d + d_1 + \cdots + d_r -n - r - 1)
\] 
and $d + d_1 + \cdots + d_r \ge n+r+1$.
By Theorem \ref{thm:totaro}, $Y$ is not universally $\CH_0$-trivial, and the non-stable-rationality is proved.

We prove the latter part.
Suppose that $d + d_1 + \cdots + d_r \ge n+r+2$ and the covering degree $m$ is not a power of $2$.
Then the line bundle $\varphi^*\mcM \subset \Omega_Y^{n-1}$ is big and we can choose $p \ne 2$ such that $p \mid m$.
Thus the assertion follows from \cite{Kol2} (see the proof of V.5.14 Theorem).
\end{proof}

%

\begin{proof}[Proof of \emph{Theorem \ref{mainthm2}}]
Let $x_0,\dots,x_n,y$ and $z$ be the homogeneous coordinates of $\mbP (1^n,a,b)$ of weight $1,\dots,a$ and $b$, respectively.
Let $n,a,b,d$ be as in Theorem \ref{mainthm2} and we set $l = d/a$, $m = d/b$.
We see that the weighted hypersurface $X$, defined over $\mbC$, is indeed nonsingular by \cite[8.1 Theorem]{IF}, since $a$, $b$ are coprime and $a \mid d$, $b \mid d$.
Moreover it is Fano since $\omega_X \cong \mcO_X (d - n - a - b)$ is anti-ample.

By the assumption (2), we can choose a prime number $p$ which divides $m$ but not $l$.
Let
\[
X = (z^m + f (x_0,\dots,x_n,y) = 0) \subset \mbP (1^n,a,b)
\]
be a weighted hypersurface of degree $d$ defined over an algebraically closed field $\K$ of characteristic $p$.
We set $Z = \mbP (1^n,a)$ and let $Z^{\circ} = Z \setminus \{(0\!:\!\cdots\!:\!0\!:\!1)\}$ be the nonsingular locus of $Z$.
Let $\pi \colon X \to Z$ be the projection, and set $X^{\circ} := \pi^{-1} (Z^{\circ})$ and $\pi^{\circ} = \pi|_{X^{\circ}}$.
Then $X^{\circ} = Z^{\circ} [\sqrt[m]{f}]$.
If $m = d/b \ne 2$, then we can assume $p \ne 2$.
In this case the restriction map
\[
H^0 (Z, \mcO_Z (d)) \to \mcO_Z (d) \otimes (\mcO_{Z,\msq}/\mfm_{\msq}^3)
\]
is surjective for any $\msq \in Z^{\circ}$ since $d/a \ge 2$.
If $m = d/b = 2$, then $p = 2$ and 
\[
H^0 (Z, \mcO_Z (d)) \to \mcO_Z (d) \otimes (\mcO_{Z,\msq}/\mfm_{\msq}^4)
\]
is surjective for any $\msq \in Z^{\circ}$ since $d/a \ge 3$ (Note that the inequality $d/a \ge 3$ follows from the assumption (3)).
Thus $f$ has only admissible critical points on $Z^{\circ}$.
Since we can assume that $f$ contains $y^l$ and $p \nmid l$, we see that $X$ is nonsingular at any point of the finite set $X \setminus X^{\circ}$.
We set $\mcM = \iota_* \mcM^{\circ}$, where $\mcM^{\circ} = {\pi^{\circ}}^* \mcQ (\mcO_{Z^{\circ}} (d), f)$.
Note that $\mcM$ is a line bundle.
By Proposition \ref{prop:resol}, there is a universally $\CH_0$-trivial resolution $\varphi \colon Y \to X$ such that $\varphi^*\mcM \inj \Omega_Y^{n-1}$.
We have 
\[
\mcM \cong \pi^* (\omega_Z \otimes \mcO_Z (d)) \cong \mcO_X (d - n - a)
\] 
and $H^0 (X,\mcM) \ne 0$ since $d \ge d-n-a$.
Thus a very general $X$, defined over $\mbC$, is not stably rational.
\end{proof}

\begin{proof}[Proof of \emph{Corollary \ref{maincor2}}]
The assertion for $X_5 \subset \mbP^5$ follows from \cite{Totaro}.
Since $X_6 \subset \mbP (1^5,2)$ and $X_8 \subset \mbP (1^5,4)$ are triple and double covers of $\mbP^4$ branched along a divisor of degree $6$ and $8$, respectively, they are not stably rational by Theorem \ref{mainthm} or \cite{CTP2}.
Finally the assertion for $X_{10} \subset \mbP (1^4,2,5)$ follows from Theorem \ref{mainthm2}.
\end{proof}


\begin{thebibliography}{99}
\bibitem{Beau}
A.~Beauville,
A very general sextic double solid is not stable rational,
Bull. Lond. Math. Soc. {\bf 48} (2016), no.~2, 321--324.

\bibitem{BK}
G.~Brown and A.~Kasprzyk,
Four-dimensional projective orbifold hypersurfaces,
Experiment. Math. {\bf 25} (2016), no.~2, 176--193.

\bibitem{CL}
A.~Chatzistamatiou and M.~Levine,
Torsion orders of complete intersections,
arXiv:1605.01913.

\bibitem{Chel1}
I.~Cheltsov,
Birationally super-rigid cyclic triple spaces,
Izv. Math. {\bf 68} (2004), no.~6, 1229--1275.

\bibitem{Chel2}
I.~Cheltsov,
Double spaces with singularities,
Sb. Math. {\bf 199} (2008), no.~1-2, 291--306.

\bibitem{CTP1}
J.-L.~Colliot-Th\'el\`ene and A.~Pirutka,
Hypersurfaces quartiques de dimension 3: non rationalit\'e stable, 
Ann. Sci. \'Ec. Norm. Sup\'er. (4) {\bf 49} (2016), no.~2, 371--397. 

\bibitem{CTP2}
J.-L.~Colliot-Th\'el\`ene and A.~Pirutka,
Rev\^etements cycliques qui ne sont pas stablement rationnels,
arXiv:1506.00420, to appear in Izvestiya RAN, Ser. Math.

\bibitem{HT}
B.~Hassett and Y.~Tschinkel,
On stable rationality of Fano threefolds and del Pezzo fibration,
arXiv:1601.07074.

\bibitem{HKT}
B.~Hassett, A.~Kresch and Y.~Tschinkel,
Stable rationality and conic bundles,
arXiv:1503.08497, to appear in Math. Ann.

\bibitem{HPT}
B.~Hassett, A.~Pirutka and Y.~Tschinkel,
Stable rationality of quadric surface bundles over surfaces,
arXiv:1603.09262.

\bibitem{IF}
A.~R.~Iano-Fletcher,
Working with weighted complete intersections,
{\it Explicit birational geometry of $3$-folds}, London Math. Soc. Lecture note Ser. {\bf 281}, Cambridge Univ. Press, Cambridge, 2000.

\bibitem{Kol1}
J.~Koll\'ar,
Nonrational hypersurfaces,
J. Amer. Math. Soc. {\bf 8} (1995), no.~1, 241--249.

\bibitem{Kol2}
J.~Koll\'ar,
{\it Rational curves on algebraic varieties},
Ergebnisse der Mathematik und ihrer Grenzgebiete, {\bf 32}, Springer-Verlag, Berlin, 1996.

\bibitem{Pir}
A.~Pirutka,
Varieties that are not stably rational, zero-cycles and unramified cohomology,
arXiv:1603.09261.

\bibitem{Puk}
A.~V.~Pukhlikov,
Birational automorphisms of double spaces with singularities,
J. Math. Sci. (New York) {\bf 85} (1997), no.~4, 2128--2141.


\bibitem{Totaro}
B.~Totaro,
Hypersurfaces that are not stably rational,
J. Amer. Math. Soc. {\bf 29} (2016), no.~3, 883--891.

\bibitem{Voisin}
C.~Voisin,
Unirational threefolds with no universal codimension 2 cycle,
Invent. Math. {\bf 201} (2015), no.~1, 207--237.
\end{thebibliography}
\end{document}